\documentclass[11pt]{article}

\usepackage[T1]{fontenc}
\usepackage[utf8]{inputenc}
\usepackage{lmodern}
\usepackage[margin=1in]{geometry}

\usepackage{amsmath,amssymb,amsthm,mathtools,mathrsfs,cite}
\usepackage{xcolor,bbm}
\usepackage{enumitem}
\usepackage{csquotes}
\usepackage[hidelinks]{hyperref}
\usepackage{cleveref}

\newtheorem{theorem}{Theorem}[section]

\newtheorem{proposition}[theorem]{Proposition}
\newtheorem{lemma}[theorem]{Lemma}
\newtheorem{corollary}[theorem]{Corollary}
\theoremstyle{definition}
\newtheorem{definition}[theorem]{Definition}
\theoremstyle{remark}
\newtheorem{remark}[theorem]{Remark}

\definecolor{revieworange}{HTML}{E87500}

\newcommand{\C}{\mathbb{C}}
\newcommand{\R}{\mathbb{R}}
\newcommand{\E}{\mathbb{E}}
\newcommand{\dd}{\,d}

\title{Solutions to 6 problems in recurrence and Van der Corput sets}
\author{%
    Sohail Farhangi\thanks{Beijing Institute of Mathematical Sciences
    and Applications (BIMSA), Beijing, China.}
    \and
    Sa\'ul Rodr\'iguez Mart\'in\thanks{Department of Mathematics,
    The Ohio State University, Columbus, Ohio, USA.}%
}
\date{\vspace{-30pt}}

\begin{document}

\maketitle
\begin{abstract}
    We present solutions to six problems concerning sets of nice recurrence and van der Corput (vdC) sets.
    We answer a question of Moreira, repeated by Fish and Skinner, by showing that every set of ergodic nice recurrence is a set of nice recurrence.
    We answer questions of Bergelson as well as Bergelson and Lesigne by showing that every set of nice recurrence is a nice vdC-set, that the family of nice vdC-sets is partition regular, and that the family of sets of nice recurrence is partition regular.
    We mention that Chapman independently proved that sets of nice recurrence are partition regular. 
    We answer a question of Kelly and L\^e, a special case of which was previously asked by Peres, and show that for every nontrivial compact metrizable group $K$, every $K$-vdC set is also a vdC set.
    Lastly, we answer a question of Farhangi, Rodr\'iguez, and Tucker-Drob and show that for any countably infinite discrete group $G$, a vdC set in $G$ is a set of operatorial recurrence in $G$.

    The initial solutions to these problems were produced by an AI system. The authors checked their correctness, simplified and reformulated them as clear, human-readable proofs, and situated the resulting contributions within the existing literature.
\end{abstract}

\begingroup
\par
\endgroup
\medskip
\section{Introduction}
The Furstenberg correspondence principle \cite{FurstenbergBook,FurstenbergsProofOfSzemeredi} gives a link between density Ramsey theory and measure-preserving dynamics that spawned the field now known as ergodic Ramsey theory.
The correspondence principle identifies questions about finite configurations appearing in positive density subsets of $\mathbb{N}$ with questions about multiple recurrence for sets of positive measure in measure-preserving systems.
The most basic instance of this is seen when working with sets of recurrence.
A subset $R\subseteq\mathbb{N}$ is a \textbf{set of measurable recurrence}
if, for every measure-preserving system\footnote{Here, $(X,\mathscr{B},\mu)$ is a
probability space and $T\colon X\to X$ is measure preserving; that is,
$\mu(T^{-1}A)=\mu(A)$ for every $A\in\mathscr{B}$.} $(X,\mathscr{B},\mu,T)$ and every
$A\in\mathscr{B}$,
\begin{equation*}
    \mu\bigl(A\cap T^{-r}A\bigr)=0
    \quad\text{for every } r\in R
    \qquad\Longrightarrow\qquad
    \mu(A)=0.
\end{equation*}
The correspondence principle tells us that $R \subseteq \mathbb{N}$ is a set of measurable recurrence if and only if for every $A \subseteq \mathbb{N}$ satisfying $\overline{d}(A) := \displaystyle\overline{\lim_{N\rightarrow\infty}}\frac{|A\cap[1,N]}{N} > 0$ there exists an $r \in R$ for which $\overline{d}(A\cap(A-r)) > 0$.
Using the pigeonhole principle, one can show that sets of differences, i.e., sets of the form $A = \{n_i-n_j\ |\ i > j\}$ for some increasing sequence $(n_i)_{i = 1}^\infty$, are sets of measurable recurrence.
The Furstenberg-S\'ark\"ozy Theorem \cite{FurstenbergsProofOfSzemeredi,SarkozyDifferenceSetsI} (see also \cite[Page~53]{ERTAnUpdate}) can be reinterpreted as the statement that the squares are a set of measurable recurrence.
In fact, Bergelson, Furstenberg, and McCutcheon \cite[Corollary 2.1]{BFM1996IPPolynomialRecurrence} observed that if $p$ is a polynomial satisfying $p(\mathbb{N}) \subseteq \mathbb{N}$ and $p(0) = 0$, and $B \subseteq \mathbb{N}$ is an IP-set\footnote{The set $B \subseteq \mathbb{N}$ is an \textbf{IP-set} if there exists a sequence $(x_n)_{n = 1}^\infty$ in $\mathbb{N}$ such that for any nonempty finite set $F \subseteq \mathbb{N}$, we have $\sum_{f \in F}x_f \in B$.}, then $p(B)$ is a set of measurable recurrence.

A set $R \subseteq \mathbb{N}$ is a \textbf{van der Corput set (vdC set)} if for every sequence $(x_n)_{n = 1}^\infty \subseteq [0,1]$ for which $(x_{n+r}-x_n\pmod{1})_{n = 1}^\infty$ is uniformly distributed for all $r \in R$, we also have that $(x_n)_{n = 1}^\infty$ is uniformly distributed.
Peres \cite[Theorem 1.1]{PeresApplicationsOfBanachLimits} showed that $R$ is a vdC set if and only if for any unitary operator $U$ on a Hilbert space $\mathcal{H}$, and any $\xi \in \mathcal{H}$,
\begin{equation*}
    \langle U^r\xi,\xi\rangle = 0
    \quad\text{for every } r\in R
\qquad\Longrightarrow\qquad
    P\xi = 0.
\end{equation*}
where $P:\mathcal{H}\rightarrow\mathcal{H}$ is the orthogonal projection onto the space of $U$-invariant vectors.
This characterization of vdC sets was independently rediscovered by Nin\v{c}evi\'c, Rabar, and Slijep\v{c}evi\'c \cite{vdCIsOperatorialRecurrent}, who called such sets operator recurrent.
We choose to follow the convention of \cite{FTD}, and refer to vdC sets as \textbf{sets of operatorial recurrence}.
Since every (invertible) measure-preserving system \((X,\mathscr B,\mu,T)\) gives rise to the so-called Koopman unitary operator \(f\mapsto f\circ T\) on \(L^2(X,\mu)\), applying operatorial recurrence to \(\xi=\mathbf 1_A\) shows that every set of operatorial recurrence is a set of measurable recurrence.
Ruzsa \cite[Page 1425]{RuzsavanderCorput} asked whether or not there exists a set of measurable recurrence that is not a vdC set, and Bourgain \cite{vdCStrongerThanRecurrence} constructed an example of such a set (see also \cite{RefiningBourgain}).
While the term vdC set is more classical (see \cite{RuzsavanderCorput,KMFvdC,PeresApplicationsOfBanachLimits,RefiningBourgain,vdCStrongerThanRecurrence,vanderCorputSetsInZ^d}), the language of operatorial recurrence more clearly highlights the relationship between sets of measurable recurrence and vdC sets.

There are many other variations of sets of recurrence and questions about the various possible implications between them.
For example, we say $R \subseteq \mathbb{N}$ is a \textbf{set of (ergodic) nice recurrence} if for every (ergodic) probability measure-preserving system $(X,\mathscr{B},\mu,T)$, every $A \in \mathscr{B}$, and every $\varepsilon > 0$, there exists $r \in R$ for which 
$\mu(A\cap T^{-r}A) \ge \mu(A)^2-\varepsilon$.
Sets of differences, the set of squares, and the polynomial images of IP-sets described above are all sets of nice recurrence.
Forrest \cite{Forrest1990Recurrence,Forrest1991RecurrenceNotStrong} constructed a set of measurable recurrence which is not a set of nice recurrence, and it follows from our own \Cref{r9ewiodfslk} that the constructions of Bourgain \cite{vdCStrongerThanRecurrence} and Mountakis \cite{RefiningBourgain} also yield sets of measurable but not nice recurrence.

It is a routine application of the ergodic decomposition to show that if $R$ is a set of recurrence for all ergodic measure-preserving systems, then it is a set of recurrence.
However, the same technique does not work in the setting of nice recurrence, as noted by the fact that Moreira \cite{Moreira2013NiceRecurrence} as well as Fish and Skinner \cite[Question 3]{FishSkinnerInverseFurstenberg} asked whether or not every set of ergodic nice recurrence is a set of nice recurrence.
We answer this question in the positive in \Cref{r9ewiodfslk}, i.e., we show that every set of ergodic nice recurrence is a set of nice recurrence.

Bergelson and Lesigne \cite[Definition 10]{vanderCorputSetsInZ^d} introduced the notion of nice vdC sets and asked about its relationship with nice recurrence.
Farhangi \cite[Theorem~5.1.6]{FarhangiThesis} proved that every nice vdC set in $\mathbb N$ is a set of nice recurrence, and Rodr\'iguez \cite[Section~2]{SaulInverseFurstenberg} extended this implication to countably infinite amenable groups.
In \Cref{r9ewiodfslk}, we prove the converse implication, proving that a subset of any countable group is a set of nice recurrence if and only if it is nice vdC.

It is a classical fact that sets of recurrence are partition regular, i.e., if $R$ is a set of recurrence and $R = R_1\bigcup R_2$, then either $R_1$ or $R_2$ must also be a set of recurrence.
Similarly, vdC sets and sets of operatorial recurrence are also partition regular.
The partition regularity of sets of recurrence follows immediately from considering products of measure-preserving systems, and the partition regularity of sets of operatorial recurrence follows from considering tensor products of Hilbert spaces and unitary operators.
The question whether the family of sets of nice recurrence is partition
regular has appeared repeatedly in the literature
\cite{vanderCorputSetsInZ^d,ERTAnUpdate,Forrest1990Recurrence,
Krause2024PointwiseErgodicTheory,DensitySchur}.
Bergelson and Lesigne also asked whether nice vdC sets are
partition regular \cite[Question~9]{vanderCorputSetsInZ^d}.
After recalling that the class of sets of nice recurrence coincides with the class of nice vdC sets, we give a positive answer to both questions by showing that sets of nice recurrence are partition regular.
This result was independently proven by Jonathan Chapman \cite{Chapman2026NiceRecurrence} without the use of AI.\footnote{We prompted the AI system to work on this problem on July 24, 2026.
At that time, Chapman's article had not yet appeared on arXiv, and we were therefore unaware of his solution.} 

Although uniform distribution was originally defined for sequences in the interval $[0,1]$, the notion extends naturally to the unit torus $\mathbb{T}$ and, more generally, to any compact group $K$.
Namely, a sequence $(x_n)_{n=1}^\infty$ in $K$ is \textbf{uniformly distributed} if, for every $f\in C(K)$,
\begin{equation*}
    \lim_{N\to\infty}\frac{1}{N}\sum_{n=1}^N f(x_n)
    =
    \int_K f\dd m_K,
\end{equation*}
where $m_K$ denotes the normalized Haar measure on $K$.

Given a compact group $K$ and a set $R \subseteq \mathbb{N}$, $R$ is a $K$-vdC set if for any sequence $(x_n)_{n = 1}^\infty$ in $K$ for which $(x_{n+r}x_n^{-1})_{n = 1}^\infty$ is uniformly distributed in $K$ for all $r \in R$, we have that $(x_n)_{n = 1}^\infty$ is itself uniformly distributed in $K$.
Peres \cite[Section 4]{PeresApplicationsOfBanachLimits} observed that for any compact abelian group $K$, every vdC set is also a $K$-vdC set, and asked when the converse holds.
Kelly and L\^e \cite[Theorem 1]{vdCSetsWithRespectToCompactGroups} independently observed that for any compact group $K$, every vdC set is also a $K$-vdC set, and asked when the converse holds.
Rodr\'iguez \cite[Proposition 7.12]{SaulInverseFurstenberg} showed that every $\mathbb{Z}/2\mathbb{Z}$-vdC set is a vdC set, thereby answering a special case of the questions of Peres, Kelly, and L\^e.
As a consequence of \Cref{thm:equivalenceForCompactvdC,thm:Sequence(KG)vdCIsMeasureTheoretic(KG)vdC}, we answer these questions and show that for every nontrivial compact metrizable group $K$, every $K$-vdC set is a vdC set (see also \Cref{rem:MetrizabilityAssumption}).

When $G$ is a countably infinite discrete group, Rodr\'iguez \cite[Definition 1.13]{SaulInverseFurstenberg} defined $R \subseteq G$ to be a vdC set if for every measure-preserving system $(X,\mathscr{B},\mu,(T_g)_{g \in G})$ and every $f \in L^\infty(X,\mu)$,
\begin{equation}
\label{Eq:MeasurableDefVdC}
    \int_X f(T_rx)\,\overline{f(x)}\,\dd\mu(x)= 0
    \quad\text{for every } r\in R
    \qquad\Longrightarrow\qquad
    \int_X f(x)\,\dd\mu(x)= 0,
\end{equation}
and he showed that this is equivalent to the standard definition involving uniform distribution when $G$ is amenable.
In this setting we define $R \subseteq G$ to be a set of operatorial recurrence if for every unitary representation $U$ of $G$ on a Hilbert space $\mathcal{H}$, and every vector $\xi \in \mathcal{H}$ satisfying $\langle U^r\xi,\xi\rangle = 0$ for all $r \in R$, we have $P\xi = 0$, where $P:\mathcal{H}\rightarrow\mathcal{H}$ is the orthogonal projection onto the space of $U$-invariant vectors.
It is a consequence of the Gaussian measure space construction (see \Cref{ssec:GMSC}) that $R$ is a set of operatorial recurrence if and only if \Cref{Eq:MeasurableDefVdC} is satisfied for every measure-preserving system $(X,\mathscr{B},\mu,(T_g)_{g \in G})$ and every $f \in L^2(X,\mu)$.
Rodr\'iguez \cite{SaulInverseFurstenberg} as well as Farhangi and Tucker-Drob \cite{FTD} observed that when $G$ is abelian, every vdC set is also a set of operatorial recurrence, but they were not able to determine whether or not this is the case in general.
We show that for every countably infinite discrete group $G$, a vdC set in $G$ is also a set of operatorial recurrence in \Cref{Thm:vdCiffOpRec} (see also \Cref{rem:NontrivialityOfLInfinityVersusL2}).\\

\noindent\textbf{Organization:} The paper is organized as follows.
In \Cref{sec:EquivalentFormsOfNiceRecurrence}, we establish several equivalent formulations of nice recurrence, including its equivalence with ergodic nice recurrence and nice vdC.
In \Cref{sec:PartitionRegularity}, we use the characterization of sets of nice recurrence (or equivalently, nice vdC sets) in terms of positive-definite functions to prove that these families of sets are partition regular.
In \Cref{sec:vdCSetsWRTCompactGroups}, we study vdC sets with respect to compact metrizable groups, first in a measure-theoretic formulation and then through uniform distribution along F\o lner
sequences.
Finally, in \Cref{sec:vdCIsOperatorRecurrent}, we recall the Gaussian measure space construction and use it to prove that every vdC set in a countably infinite group is a set of operatorial recurrence.\\

\noindent\textbf{AI use acknowledgement:} The initial solutions to the six problems presented in this paper were obtained using the large language model GPT Pro 5.6 sol.
The authors first examined each proposed argument in detail and verified its mathematical correctness.
They then substantially simplified and reformulated the arguments, organizing them as clear, human-readable proofs.
When writing this paper, ChatGPT was again used to rewrite several parts more clearly and to check the final paper for grammatical and mathematical correctness.
The authors take full responsibility for the correctness of all statements and proofs in this paper.
%%%%%%%%%%%%%%%%%%%%%%%%%%%%%%%%%%%%%%%%%%%%%%%%%%%%%%%%%%%%%%%%%%%%%%%%%%%%%%%%%%%%%%%%%%%%%%%%%%%%%%%%%%%%%%%%%%%%%%%%
\section{Equivalent forms of nice recurrence}\label{sec:EquivalentFormsOfNiceRecurrence}

Throughout this paper, $G$ denotes a countably infinite discrete group with identity element $e$, and $H\subseteq G$.
A measure-preserving system (m.p.s.) is denoted by $(X,\mathscr{B},\mu,(T_g)_{g\in G})$, where $(X,\mathscr{B},\mu)$ is a standard probability space and $T$ is a measure-preserving action of $G$.
We will use $T$ both for the measure-preserving action of $G$ on $(X,\mathscr{B},\mu)$ and for the Koopman representation of $G$ on $L^2(X,\mu)$ given by $T_gf = f\circ T_{g^{-1}}$.

Our next result is a generalization of \cite[Proposition 4.1]{BFW} (see also \cite[Theorem 10.1]{BTZ}) in the setting of countably infinite groups.

\begin{proposition}
\label{CorrFunctionToCorrSet}
Let $G$ be a countably infinite group, let $(X,\mathscr{B},\mu,(T_g)_{g\in G})$ be an m.p.s. and let $f:X\to[0,1]$ be measurable.
Then there exist an m.p.s. $(Y,\mathcal{C},\nu,(S_g)_{g\in G})$ and a set $C\in\mathcal{C}$ such that, for all distinct $g_1,\dots,g_k\in G$, we have
\begin{equation*}
\nu\left(\bigcap_{i=1}^k S_{g_i}^{-1}C\right)
=\int_X \prod_{i=1}^k f(T_{g_i}x)\dd\mu(x).
\end{equation*}
If the original m.p.s. is ergodic, then the constructed m.p.s. is ergodic as well.
\end{proposition}

\begin{proof}
Let $m$ denote Lebesgue measure on $[0,1]$, equip $[0,1]^G$ with the
product measure $m^G$, and let $G$ act on $[0,1]^G$ by the right
Bernoulli shift
\begin{equation*}
(\Phi_g t)_h=t_{hg}\qquad(g,h\in G).
\end{equation*}
Thus $\Phi_g\Phi_h=\Phi_{gh}$.  Set
\begin{equation*}
(Y,\mathcal{C},\nu)
=\bigl(X\times[0,1]^G,\mathscr{B}\otimes\mathscr{B}([0,1])^{\otimes G},
\mu\otimes m^G\bigr)
\end{equation*}
and let $S_g=T_g\times\Phi_g$.  If $e$ denotes the identity of $G$, define
\begin{equation*}
C=\{(x,(t_h)_{h\in G})\in X\times[0,1]^G:t_e\leq f(x)\}.
\end{equation*}
For distinct $g_1,\dots,g_k\in G$, Fubini's theorem and the independence
of the coordinates indexed by $g_1,\dots,g_k$ give
\begin{align*}
\nu\left(\bigcap_{i=1}^k S_{g_i}^{-1}C\right)
&=\nu\bigl(\{(x,t)\in X\times[0,1]^G:t_{g_i}\leq f(T_{g_i}x)
\text{ for }i=1,\dots,k\}\bigr)\\
&=\int_X m^G\bigl(\{t\in[0,1]^G:t_{g_i}\leq f(T_{g_i}x)
\text{ for }i=1,\dots,k\}\bigr)\dd\mu(x)\\
&=\int_X\prod_{i=1}^k f(T_{g_i}x)\dd\mu(x).
\end{align*}
Finally, the Bernoulli shift of an infinite group is weakly mixing.
Hence its
product with any ergodic measure-preserving $G$-action is ergodic.
\end{proof}

We are ready to prove the following theorem: 

\begin{theorem}\label{r9ewiodfslk}
Let $G$ be a countably infinite group and let $R\subseteq G$ be nonempty.
The following are equivalent:
\begin{enumerate}[label=\alph*)]
    \item (Nice recurrence)\label{r9ewiodfslk1} For every m.p.s. $(X,\mathscr{B},\mu,(T_g)_{g\in G})$, every
    $A\in\mathscr{B}$, and every $\varepsilon>0$, there is $r\in R$ such that
\begin{equation*}
\mu(A\cap T_r^{-1}A)\geq\mu(A)^2-\varepsilon.
\end{equation*}
    \item (Ergodic nice recurrence)\label{r9ewiodfslk1Erg} The same statement
    holds for every ergodic m.p.s.
    
    \item (Nice operatorial recurrence)\label{r9ewiodfslk2} For every unitary representation $\pi$ of
    $G$ on a Hilbert space $\mathcal{H}$, every $f\in\mathcal{H}$, and every $\varepsilon > 0$, there is $r \in R$ for which $|\langle\pi(r)f,f\rangle| \ge \|Pf\|^2-\varepsilon$, where $P:\mathcal{H}\rightarrow\mathcal{H}$ is the orthogonal projection onto the space of $\pi$-invariant vectors.

    \item\label{r9ewiodfslkPD} Every real-valued positive-definite
    function $\varphi\colon G\to\mathbb{R}$ satisfies $\sup\{\varphi(r);r\in R\}\geq0$. 
    
    \item (Nice vdC)\label{r9ewiodfslk5} For every m.p.s. $(X,\mathscr{B},\mu,(T_g)_{g\in G})$, every
    $f\in L^\infty(X,\mu)$, and every $\varepsilon>0$, there is $r\in R$
    such that
\begin{equation*}
\left|\int_X f(T_{r^{-1}}x)\overline{f(x)}\dd\mu(x)\right|
\geq\left|\int_Xf\dd\mu\right|^2-\varepsilon.
\end{equation*}
\end{enumerate}
If, in addition, $G$ is abelian, the preceding conditions are also equivalent to the following:
\begin{enumerate}[resume, label=\alph*)]
    \item (Nice FC$^+$)\label{r9ewiodfslk4} For every probability measure $\tau$ on $\widehat G$, we have
\begin{equation*}
\sup_{r\in R}\left|\widehat\tau(r)\right|
\geq\tau\left(\{1_{\widehat G}\}\right).
\end{equation*}
\end{enumerate}
\end{theorem}

\begin{proof}
The implication
\ref{r9ewiodfslk1}$\implies$\ref{r9ewiodfslk1Erg} is immediate.

We next prove the contrapositive of
\ref{r9ewiodfslk1Erg}$\implies$\ref{r9ewiodfslkPD}. Suppose that
\ref{r9ewiodfslkPD} fails. If $R$ is empty, then
\ref{r9ewiodfslk1Erg} plainly fails as well, so we may suppose that
$R$ is nonempty.
There exist a real-valued positive-definite function $\varphi\colon G\to\mathbb{R}$ and $\varepsilon_0>0$ such that
\begin{equation*}
\varphi(r)\leq-\varepsilon_0
\qquad\text{for every }r\in R\text{ and }\varphi(e) = 1.
\end{equation*}

The ergodic realization theorem for positive-definite functions
\cite[Theorem~2.8]{FTD} gives an ergodic m.p.s.
$(X,\mathscr{B},\mu,(T_g)_{g\in G})$ and a real-valued function
$f\in L^2(X,\mu)$ such that $\langle f,f\circ T_g\rangle=\varphi(g)$ for every $g\in G$.
Choose $\eta>0$ such that $\eta(2+\eta)<\varepsilon_0/2$, and choose a real-valued function
$f'\in L^\infty(X,\mu)$ satisfying $\|f-f'\|_2<\eta$. Since
$\|f'\|_2<1+\eta$, for every $g\in G$ we have
\begin{align*}
\left|\langle f',f'\circ T_g\rangle-\varphi(g)\right|
&\leq\|f'-f\|_2\bigl(\|f'\|_2+\|f\|_2\bigr)<\frac{\varepsilon_0}{2}.
\end{align*}
Set $F=f'-\int_Xf'\dd\mu$.
Then, for every $r\in R$,
\begin{equation*}
\langle F,F\circ T_r\rangle
=\langle f',f'\circ T_r\rangle
-\left(\int_Xf'\dd\mu\right)^2
<-\frac{\varepsilon_0}{2}.
\end{equation*}
Choose $\delta>0$ small enough that $h=\frac12+\delta F$ takes values in $[0,1]$.
We have $\int_Xh\dd\mu=\frac12$, and, for every $r\in R$,
\begin{align*}
\int_Xh(T_{r^{-1}}x)h(x)\dd\mu(x)
&=\frac14+\delta^2\langle F,F\circ T_r\rangle < \left(\int_Xh\dd\mu\right)^2-\frac{\delta^2\varepsilon_0}{2}.
\end{align*}
Notice that $e\notin R$, since $\varphi(e)=1$.
Thus \Cref{CorrFunctionToCorrSet} gives an ergodic m.p.s. $(Y,\mathcal{C},\nu,(S_g)_{g\in G})$ and a set
$C\in\mathcal{C}$ such that $\nu(C)=1/2$ and, for every $r\in R$,
\begin{equation*}
\nu(C\cap S_r^{-1}C)
<\nu(C)^2-\frac{\delta^2\varepsilon_0}{2}.
\end{equation*}
Consequently, \ref{r9ewiodfslk1Erg} fails.

We now prove
\ref{r9ewiodfslkPD}$\implies$\ref{r9ewiodfslk2}. Let $\pi$ be a
unitary representation of $G$ on a Hilbert space $\mathcal{H}$, let $f\in\mathcal{H}$, and write $f=Pf+f_0$.
The function $g\longmapsto\operatorname{Re}\langle\pi(g)f_0,f_0\rangle$ is real-valued and positive definite.
Hence, for every $\varepsilon>0$, there is $r\in R$ such that $\operatorname{Re}\langle\pi(r)f_0,f_0\rangle>-\varepsilon$.
Since the subspace of vectors invariant under $\pi$ and its orthogonal complement are both
$\pi$-invariant,
\begin{equation*}
\left|\langle\pi(r)f,f\rangle\right| \geq\operatorname{Re}\langle\pi(r)f,f\rangle = \|Pf\|^2+
\operatorname{Re}\langle\pi(r)f_0,f_0\rangle > \|Pf\|^2-\varepsilon.
\end{equation*}

We next prove \ref{r9ewiodfslk2}$\implies$\ref{r9ewiodfslk5}.
Let $(X,\mathscr{B},\mu,(T_g)_{g\in G})$ be an m.p.s., let $f\in L^\infty(X,\mu)$, and let $P$ be the orthogonal
projection from $L^2(X,\mu)$ onto the subspace of invariant functions.
Applying \ref{r9ewiodfslk2} to the Koopman representation gives, for
every $\varepsilon>0$, an $r\in R$ such that
\begin{equation*}
\left|\int_Xf(T_{r^{-1}}x)\overline{f(x)}\dd\mu(x)\right| \geq\|Pf\|_2^2-\varepsilon \geq\left|\int_Xf\dd\mu\right|^2-\varepsilon.
\end{equation*}

Finally, taking $f=1_A$ in
\ref{r9ewiodfslk5} gives \ref{r9ewiodfslk1}.
This completes the proof for an arbitrary countably infinite group.

Suppose now that $G$ is abelian. We first prove
\ref{r9ewiodfslk5}$\implies$\ref{r9ewiodfslk4}.
Let $\tau$ be a probability measure on $\widehat G$.
If $a := \tau(\{1_{\widehat{G}}\}) \in \{0,1\}$, then the desired result is immediate, so let us assume that $a \in (0,1)$.
Consider $\tau_a := (1-a)^{-1}(\tau-a\delta_{1_{\widehat{G}}})$.
We use \cite[Lemma~3.4]{FTD} to obtain an m.p.s. $(X,\mathscr{B},\mu,(T_g)_{g\in G})$ and $u \in L^\infty(X,\mu)$ such that $\widehat{\tau_a}(g) = \langle T_gu,u\rangle$ for all $g \in G$, and $0 = \tau_a(\{1_{\widehat{G}}\}) = \int_Xu\dd\mu$.
Let $u_a = \sqrt{a}\mathbbm{1}_X+\sqrt{1-a}u$, and observe that $\langle T_gu_a,u_a\rangle = a+(1-a)\langle T_gu,u\rangle = \widehat{\tau}(g)$ for all $g \in G$, and $\int_Xu_a\dd\mu = \sqrt{a}$.
We now see that

\begin{equation*}
    \sup_{r \in R}\left|\widehat{\tau}(r)\right| = \sup_{r \in R}\left|\langle T_ru_a,u_a\rangle\right| \ge \left|\int_Xu_a\dd\mu\right|^2 = a = \tau(\{1_{\widehat{G}}\}).
\end{equation*}

Conversely, suppose that \ref{r9ewiodfslk4} holds.
Let $(X,\mathscr{B},\mu,(T_g)_{g\in G})$ be an m.p.s., let $f\in L^\infty(X,\mu)$, and let $P$ denote the orthogonal projection from $L^2(X,\mu)$ onto the invariant functions.
Assuming without loss of generality that $\|f\|_2 = 1$, we apply the Bochner--Herglotz theorem \cite[Section~1.4.3]{RudinFourierGroups} to obtain a probability measure $\tau_f$ on $\widehat G$ such that
\begin{equation*}
\widehat{\tau_f}(g)
=\int_Xf(T_gx)\overline{f(x)}\dd\mu(x)\text{ for every }g\in G,\text{ and }\tau_f(\{1_{\widehat{G}}\}) = \|Pf\|_2^2 \ge \left|\int_Xf\dd\mu\right|^2.
\qedhere
\end{equation*}
\end{proof}

\begin{remark}\label{rem:OtherEquivalences}
    Theorems 1.4 and A.1 of \cite{FTD} give a long list of equivalent formulations of vdC sets.
    Almost all of these formulations have an analogue for nice vdC sets, but we do not pursue this here for the sake of brevity.
\end{remark}

%%%%%%%%%%%%%%%%%%%%%%%%%%%%%%%%%%%%%%%%%%%%%%%%%%%%%%%%%%%%%%%%%%%%%%%%%%%%%%%%%%%%%%%%%%%%%%%%%%%%%%%%%%%%%%%%%%%%%%%%%%%%%%%%%%%
\section{Partition regularity of nice vdC sets and sets of nice recurrence}\label{sec:PartitionRegularity}
In this section we prove that if the union of two subsets $A,B$ of a countably infinite group $G$ is nice vdC, then either $A$ or $B$ is nice vdC.
The proof relies on the characterization of nice vdC sets via positive-definite functions given in \Cref{r9ewiodfslk}\ref{r9ewiodfslkPD}.

\begin{lemma}
\label{r4ew9dsopl9erwofpds}
Let $G$ be a group.  If $\phi_1,\phi_2\colon G\to\mathbb{R}$ are positive
definite and $p\in\mathbb{R}[x,y]$ has nonnegative coefficients, then the function
\begin{equation*}
g\longmapsto p\left(\phi_1(g),\phi_2(g)\right)
\end{equation*}
is positive definite.
\end{lemma}

\begin{proof}
By the GNS construction \cite[Theorem~3.20]{Folland}, a function
$\phi\colon G\to\mathbb C$ is positive definite if and only if there exist
a unitary representation $\pi$ of $G$ on a Hilbert space $\mathcal H$ and
a vector $\xi\in\mathcal H$ such that
\[
\phi(g)=\langle\pi(g)\xi,\xi\rangle
\qquad\text{for every }g\in G.
\]
Consequently, if $\lambda\geq0$ and $\phi_1$ and $\phi_2$ are positive definite, then so are $\lambda\phi_1$, $\phi_1+\phi_2$, $\phi_1\phi_2$ (the last two follow by considering direct sums and
tensor products of unitary representations).
Since every nonnegative constant function is positive definite, the
result follows.
\end{proof}

\begin{lemma}
\label{4re9wfoidslk}
For all positive real numbers $0<\varepsilon<M$ there exist $\delta>0$ and
$K\in\mathbb{N}$ such that, whenever $x,y\in[-M,M]$ and
either $x<-\varepsilon$ or $y<-\varepsilon$, we have
\begin{equation*}
x(2M+y)^K+y(2M+x)^K\le-\delta.
\end{equation*}
\end{lemma}

\begin{proof}
Choose an even integer $K\geq2$ large enough that
\begin{equation*}
M(2M-\varepsilon)^K\leq\frac{\varepsilon}{2}(2M)^K,
\end{equation*}
and set $\delta=\varepsilon M^K$.
By symmetry, we may assume that
$x<-\varepsilon$, so that $\left|2M+x\right|<2M-\varepsilon$.  
If $y\leq0$, then
\begin{equation*}
x(2M+y)^K+y(2M+x)^K
\leq x(2M+y)^K
<-\varepsilon M^K=-\delta.
\end{equation*}
If $y\geq0$, then $2M+y\geq2M$, and hence
\begin{equation*}
x(2M+y)^K+y(2M+x)^K
<-\varepsilon(2M)^K+M(2M-\varepsilon)^K
\leq-\frac{\varepsilon}{2}(2M)^K
=-2^{K-1}\delta.\qedhere
\end{equation*}
\end{proof}

\begin{theorem}\label{thm:NicevdCIsPartitionRegular}
Let $G$ be a countably infinite group.
If $A,B\subseteq G$ are not nice vdC, then $A\cup B$ is not nice vdC.
\end{theorem}

\begin{proof}
Since $A$ and $B$ are not nice vdC, there are real positive-definite functions
$\phi_A,\phi_B\colon G\to\mathbb{R}$ and $\varepsilon>0$ such that
\begin{equation*}
\phi_A(a)<-\varepsilon\quad\text{for every }a\in A,
\qquad\text{and}\qquad
\phi_B(b)<-\varepsilon\quad\text{for every }b\in B.
\end{equation*}
Choose
\begin{equation*}
M>\max\left\{\|\phi_A\|_\infty,\|\phi_B\|_\infty\right\}.
\end{equation*}
By \Cref{4re9wfoidslk}, there exist $K\in\mathbb{N}$ and $\delta>0$ such that the function $\phi\colon G\to\mathbb{R}$ defined by
\begin{equation*}
\phi(g)
=\phi_A(g)\left(2M+\phi_B(g)\right)^K
+\phi_B(g)\left(2M+\phi_A(g)\right)^K
\end{equation*}
satisfies $\phi(g)\le-\delta$ for every $g\in A\cup B$.
Moreover, $\phi$ is positive definite by \Cref{r4ew9dsopl9erwofpds}.
Therefore, $A\cup B$ is not nice vdC.
\end{proof}

\begin{corollary}
Let $G$ be a countably infinite group, let $D\subseteq G$ be nice vdC, and suppose that
\begin{equation*}
D=D_1\cup\cdots\cup D_m.
\end{equation*}
Then at least one of the sets $D_1,\dots,D_m$ is nice vdC.
\end{corollary}

\begin{proof}
This follows from the preceding theorem by contraposition and induction on
$m$.
\end{proof}

\begin{remark}\label{rem:PartitionRegularity}
A single product of the two positive-definite witnesses does not suffice:
if both witnesses are negative, their product is positive.  The polynomial
used above instead acts as a logical ``or'': it is uniformly negative when
either input is uniformly negative, while its nonnegative coefficients
ensure that positive definiteness is preserved.
Recalling that \Cref{r9ewiodfslk} shows that nice recurrence is equivalent to nice vdC, we see that sets of nice recurrence are also partition regular.
As mentioned in the introduction, this answers a question of Bergelson, and was independently proven by Chapman without the use of AI.
\end{remark}

%%%%%%%%%%%%%%%%%%%%%%%%%%%%%%%%%%%%%%%%%%%%%%%%%%%%%%%%%%%%%%%%%%%%%%%%%%%%%%%%%%%%%%%%%%%%%%%%%%%%%%%%%%%%%%%%%%%%%%%%
\section{Van der Corput sets with respect to compact groups}\label{sec:vdCSetsWRTCompactGroups}

In this section we compare ordinary vdC sets in $G$ with two notions associated with a compact (metrizable) group $K$.
In \Cref{thm:equivalenceForCompactvdC}, we show that ordinary vdC sets are precisely the measure-theoretic $(K,G)$-vdC sets.
When $G$ is amenable and $\mathcal F$ is a left-F\o lner sequence, \Cref{thm:Sequence(KG)vdCIsMeasureTheoretic(KG)vdC} shows that the measure-theoretic notion is also equivalent to the formulation in terms of $\mathcal F$-uniform distribution of $K$-valued sequences.
Taking $G=\mathbb Z$ and $F_N=\{1,\ldots,N\}$ recovers the notion of a $K$-vdC set introduced in the Introduction.

\subsection{Preliminaries}
For amenable groups, the usual definition of a vdC set in terms of F\o lner averages is equivalent to the following dynamical condition; see \cite[Theorem~1.12]{SaulInverseFurstenberg} and also \cite[Theorem~1.1]{FTD}. 
For arbitrary countably infinite groups, \Cref{def:vdc} can be taken as the definition of vdC sets, as in \cite[Definition~1.13]{SaulInverseFurstenberg}.%; moreover, it is one of the equivalent characterizations of operatorial recurrence given in \cite[Theorem~4.2]{FTD}.

\begin{definition}\label{def:vdc}
We say that $H$ is \emph{vdC in $G$} if the following holds: whenever
$(X,\mathscr{B},\mu,(T_g)_{g\in G})$ is an m.p.s., $f\colon X\to\C$ is a
bounded measurable function, and
\begin{equation*}
\int_X f(T_hx)\overline{f(x)}\dd\mu(x)=0\text{ for all }h\in H\text{, then }\int_X f\dd\mu=0.
\end{equation*}
\end{definition}

Notice that every $H\subseteq G$ containing the identity $e$ is vdC in $G$: taking $h=e$
in the hypothesis gives $\int_X|f|^2\dd\mu=0$.

We use on $\C^d$ the inner product
$\langle z,w\rangle=w^*z$, which is linear in the first variable.

\begin{lemma}[Finite-dimensional vdC condition]\label{lem:FinDimCharvdC}%ChatGPT suggests shortening the label to fix some overflow error
Assume that $H$ is vdC in $G$.
Let $(X,\mathscr{B},\mu,(T_g)_{g\in G})$ be an m.p.s., and let
$V\colon X\to\C^d$ be bounded and measurable.
If
\begin{equation*}
\int_X\left\langle V(T_hx),V(x)\right\rangle\dd\mu(x)=0
\textup{ for all }h\in H,
\end{equation*}
then
\begin{equation*}
\int_XV\dd\mu=
\begin{pmatrix}
0\\\vdots\\0
\end{pmatrix}.
\end{equation*}
\end{lemma}

\begin{proof}
The case $d=1$ is precisely \Cref{def:vdc}, so assume that $d\geq2$.
We argue by contradiction.
Let $(e_1,\ldots,e_d)$ be an orthonormal basis of $\C^d$.
After multiplying $V$ by a nonzero scalar and composing it with a unitary transformation, we may assume that
\begin{equation*}
\int_XV\dd\mu=e_1.
\end{equation*}
Let $\xi$ be a random vector uniformly distributed on the finite set with $2(d-1)$ elements,
\begin{equation*}
Y=\left\{e_1\pm\sqrt{d-1}\,e_j:j=2,\ldots,d\right\}\subseteq\mathbb{C}^d.
\end{equation*}
By direct computation we verify that the means of the random variables $\xi$ and $\xi\xi^*$ are
\begin{equation*}
\E\xi=e_1
\qquad\text{and}\qquad
\E(\xi\xi^*)=I_d,
\end{equation*}
where $I_d$ is the $d\times d$ identity matrix. Consider $\Omega=X\times Y$ with the product probability measure $\mu_\Omega$ and the action $S_g=T_g\times\operatorname{Id}_Y$.
Define a map $q:\Omega\to\mathbb{C}$ by
\begin{equation*}
q(x,\xi)=\xi^*V(x)=\langle V(x),\xi\rangle.
\end{equation*}
For every $h\in H$, we have
\begin{align*}
\int_\Omega q(S_h(x,\xi))\overline{q(x,\xi)}\dd\mu_\Omega(x,\xi)
&=\int_X\E\left(\xi^*V(T_hx)V(x)^*\xi\right)\dd\mu(x)\\
&=\int_XV(x)^*\E(\xi\xi^*)V(T_hx)\dd\mu(x)\\
&=\int_X\left\langle V(T_hx),V(x)\right\rangle\dd\mu(x)=0.
\end{align*}
On the other hand,
\begin{equation*}
\int_\Omega q\dd\mu_\Omega
=(\E\xi)^*\int_XV\dd\mu=e_1^*e_1=1,
\end{equation*}
contradicting the vdC property of $H$.
\end{proof}

We shall also need the following convenient form of a witness to the failure of the vdC property.
The next lemma is a measure-theoretic version of \cite[Proposition~7.14]{SaulInverseFurstenberg}; in this formulation, no amenability assumption on $G$ is required.

\begin{lemma}[Sign-valued witness]\label{lem:sign-witness}%ChatGPT suggests shortening the label to fix some overflow error
If $H$ is not vdC in $G$, then there exist an m.p.s.
$(X,\mathscr{B},\mu,(T_g)_{g\in G})$ and a measurable function
$s\colon X\to\{-1,1\}$ such that
\begin{equation*}
\int_Xs(T_hx)s(x)\dd\mu(x)=0\textup{ for all }h\in H,\textup{ but }
\int_Xs\dd\mu\neq0.
\end{equation*}
\end{lemma}

\begin{proof}
By the failure of \Cref{def:vdc}, there exist an m.p.s. and a bounded complex-valued function $f$ whose correlations vanish on $H$ but whose integral is nonzero.
Set
\begin{equation*}
V=(\operatorname{Re}f,\operatorname{Im}f)\colon X\to\R^2\subseteq\mathbb{C}^2.
\end{equation*}
Then, for every $h\in H$,
\begin{align*}
\int_X\langle V(T_hx),V(x)\rangle\dd\mu(x)
&=\operatorname{Re}\int_Xf(T_hx)\overline{f(x)}\dd\mu(x)=0,
\end{align*}
whereas
\begin{equation*}
\int_XV\dd\mu
=\left(\operatorname{Re}\int_Xf\dd\mu,
        \operatorname{Im}\int_Xf\dd\mu\right)\neq(0,0).
\end{equation*}
Since $V$ and its integral are real-valued, the argument in the proof of \Cref{lem:FinDimCharvdC} produces, on a two-point extension of the system, a bounded real-valued function $u$ with zero correlations on $H$ and nonzero integral.
After rescaling, we may assume that $u$ takes values in $[-1,1]$.

Let $([-1,1]^G,m^G,(\Phi_g)_{g\in G})$ be the right Bernoulli shift, where $m$ is normalized Lebesgue measure on $[-1,1]$ and
$(\Phi_gt)_k=t_{kg}$.  On the product space $X\times[-1,1]^G$ define
\begin{equation*}
s(x,t)=
\begin{cases}
1,&t_e\leq u(x),\\
-1,&t_e>u(x).
\end{cases}
\end{equation*}
For each $x\in X$ and $h\in H$, we have
\begin{equation*}
\int_{[-1,1]^G}s\left(T_hx,(t_{gh})_{g\in G}\right)\dd m^G(t)
=u(T_hx).
\end{equation*}
Since $H$ is not vdC in $G$, the preceding observation shows that $e\notin
H$.  Thus the coordinates $t_e$ and $t_h$ are independent for every
$h\in H$.  Therefore,
\begin{align*}
\int s\left(T_hx,(t_{gh})_{g\in G}\right)s(x,t)\dd(\mu\times m^G)(x,t)&=\int_Xu(T_hx)u(x)\dd\mu(x)=0,\\
\int s(x,t)\dd(\mu\times m^G)(x,t)&=\int_Xu(x)\dd\mu(x)\neq0\qedhere
\end{align*}
\end{proof}

%\subsection{\texorpdfstring{$(K,G)$}{(K,G)}-vdC sets}
\subsection{Measure-theoretic \texorpdfstring{$K$}{K}-vdC sets}

For the rest of this section, let $K$ be a nontrivial compact metrizable group, equipped with its Borel
$\sigma$-algebra, and let $m_K$ denote its normalized Haar measure.
The definition of $(K,G)$-vdC below is the measure-theoretic analogue of the sequence definition from the introduction: 
Haar distribution of $F(T_hx)F(x)^{-1}$ replaces uniform distribution of $(x_{hg}x_g^{-1})_{g\in G}$.

\begin{definition}
We say that $H$ is \emph{measure-theoretic $(K,G)$-vdC} if the following
holds: whenever $(X,\mathscr{B},\mu,(T_g)_{g\in G})$ is an m.p.s.,
$F\colon X\to K$ is a Borel measurable map, and
\begin{equation*}
x\longmapsto F(T_hx)F(x)^{-1}
\end{equation*}
is Haar distributed for every $h\in H$, the map $F$ itself is Haar
distributed, i.e., $F_*\mu = m_K$.
\end{definition}

The following characterization allows us to identify Haar measure through the vanishing of matrix coefficients of finite-dimensional unitary representations.

\begin{proposition}[Fourier characterization of Haar measure]
\label{prop:HaarFourierCharacterization}
Let $K$ be a compact metrizable group, let $m_K$ be its normalized Haar
measure, and let $\nu$ be a Borel probability measure on $K$.
Then $\nu=m_K$ if and only if
\begin{equation}
\label{eq:HaarFourierCharacterization}
\int_K\langle\xi,U(k)\eta\rangle\dd\nu(k)=0
\end{equation}
for every $d\in\mathbb{N}$, $\xi,\eta\in\C^d$ and every nontrivial irreducible continuous unitary representation
$U\colon K\to\mathcal U(\C^d)$.
\end{proposition}

\begin{proof}
If $\nu=m_K$, \eqref{eq:HaarFourierCharacterization} follows from the
orthogonality of the coefficients of $U$ to the trivial representation
\cite[Theorem~5.8]{Folland}.  Conversely, suppose that
\eqref{eq:HaarFourierCharacterization} holds.  By the Peter--Weyl theorem (see \cite[Theorem~5.12]{Folland}),
the constant function $1$, together with the coefficient functions $
k\longmapsto\langle\xi,U(k)\eta\rangle$
appearing in \eqref{eq:HaarFourierCharacterization}, has uniformly dense
linear span in $C(K)$.  It follows that every
continuous function $f\colon K\to\mathbb{C}$ has the same integral with
respect to $\nu$ and $m_K$.  Since both measures are regular, this implies
that $\nu=m_K$.
\end{proof}

\begin{theorem}\label{thm:equivalenceForCompactvdC}
For every countably infinite group $G$ and every nontrivial compact metrizable group
$K$,
\begin{equation*}
H\text{ is vdC in }G
\quad\Longleftrightarrow\quad
H\text{ is measure-theoretic }(K,G)\text{-vdC}.
\end{equation*}
\end{theorem}

\begin{proof}
Suppose first that $H$ is vdC in $G$.  Let
$(X,\mathscr{B},\mu,(T_g)_{g\in G})$ be an m.p.s.,
and let $F\colon X\to K$ be measurable.  Assume that $
F(T_hx)F(x)^{-1}$
is Haar distributed for every $h\in H$. 

Let $U\colon K\to\mathcal{U}(\C^d)$ be a nontrivial irreducible continuous
unitary representation.  For every $\xi\in\C^d$ and $h\in H$, we have
\begin{align*}
\int_X\left\langle
U(F(T_hx)^{-1})\xi,U(F(x)^{-1})\xi
\right\rangle\dd\mu(x)
&=\int_X\left\langle
U\left(F(x)F(T_hx)^{-1}\right)\xi,\xi
\right\rangle\dd\mu(x)\\
&=\int_K\langle U(k^{-1})\xi,\xi\rangle\dd m_K(k)\\
&=\int_K\langle U(k)\xi,\xi\rangle\dd m_K(k)=0.
\end{align*}
Here we used the fact that the inverse of a Haar-distributed random variable
is also Haar distributed.
By the finite-dimensional vdC
condition, \Cref{lem:FinDimCharvdC},
\begin{equation*}
\int_XU(F(x)^{-1})\xi\dd\mu(x)=0.
\end{equation*}
Consequently, for every $\xi,\eta\in\C^d$,
\begin{equation*}
\int_X\langle\xi,U(F(x))\eta\rangle\dd\mu(x)
=\int_X\langle U(F(x)^{-1})\xi,\eta\rangle\dd\mu(x)=\left\langle\int_XU(F(x)^{-1})\xi\dd\mu(x),\eta\right\rangle=0.
\end{equation*}
Consequently, by \Cref{prop:HaarFourierCharacterization}, the distribution of $F$ is
$m_K$.  Hence $H$ is measure-theoretic $(K,G)$-vdC.

Conversely, suppose that $H$ is not vdC in $G$.  By
\Cref{lem:sign-witness}, there exist an m.p.s.
\begin{equation*}
(X,\mathscr{B},\mu,(T_g)_{g\in G})
\end{equation*}
and a measurable $s\colon X\to\{-1,1\}$ such that
\begin{equation}\label{eq:sign-witness}
\int_Xs(T_hx)s(x)\dd\mu(x)=0\textup{ for all }h\in H,\textup{ but }
m:=\int_Xs\dd\mu\neq0.
\end{equation}

Using that $K$ is nontrivial, we choose a Borel measurable function $\psi\colon K\to[-1,1]$ which
is not $m_K$-almost everywhere zero and satisfies
\begin{equation*}
\int_K\psi\dd m_K=0.
\end{equation*}
 For
$\sigma\in\{-1,1\}$, define a probability measure $\kappa_\sigma$ on $K$ by
\begin{equation}\label{eq:kappa}
\dd\kappa_\sigma=(1+\sigma\psi)\dd m_K.
\end{equation}

Set $Y=X\times K^G$.  For each $x\in X$, we define a probability measure $\nu_x$ on $K^G$ by
\begin{equation}\label{eq:fiber-measure}
\nu_x=\bigotimes_{g\in G}\kappa_{s(T_gx)},
\end{equation}
which exists because $G$ is countable.  For each cylinder
set $C\subseteq K^G$, the map $x\mapsto\nu_x(C)$ is measurable; a
monotone-class argument gives the same conclusion for every Borel set in $K^G$.
Thus $x\mapsto\nu_x$ is a probability kernel from $X$ to $K^G$.
Applying \cite[Theorem~6.11]{Cinlar} with $E=X$, $F=K^G$, and kernel
$x\mapsto\nu_x$, we obtain a probability measure $\nu$ on $Y$ characterized
by
\begin{equation}\label{eq:nu}
\int_YA\dd\nu
=\int_X\int_{K^G}A(x,\mathbf{k})\dd\nu_x(\mathbf{k})\dd\mu(x)
\end{equation}
for every bounded measurable function $A$ on $Y$.

For $h\in G$, define
\begin{equation*}
R_h((k_g)_{g\in G})=(k_{gh})_{g\in G}
\end{equation*}
on $K^G$, and define
\begin{equation*}
S_h\left(x,(k_g)_{g\in G}\right)
=\left(T_hx,R_h((k_g)_{g\in G})\right).
\end{equation*}
Since $R_gR_h=R_{gh}$, the maps $(S_h)_{h\in G}$ form a $G$-action.
Moreover, $R_h$ merely permutes the product coordinates, so
$(R_h)_*\nu_x$ is a product measure.  Its $g$-th marginal is
$\kappa_{s(T_{gh}x)}=\kappa_{s(T_gT_hx)}$; hence
\begin{equation*}
(R_h)_*\nu_x=\nu_{T_hx}.
\end{equation*}
Therefore, for every bounded measurable function $A$ on $Y$,
\begin{align*}
\int_YA\circ S_h\dd\nu
&=\int_X\int_{K^G}A(T_hx,R_h\mathbf{k})
  \dd\nu_x(\mathbf{k})\dd\mu(x)\\
&=\int_X\int_{K^G}A(T_hx,\mathbf{k})
  \dd\nu_{T_hx}(\mathbf{k})\dd\mu(x)\\
&=\int_X\int_{K^G}A(x,\mathbf{k})
  \dd\nu_x(\mathbf{k})\dd\mu(x)\\
&=\int_YA\dd\nu.
\end{align*}
The third equality follows from the $T_h$-invariance of $\mu$.
Thus $(Y,\nu,(S_h)_{h\in G})$ is an m.p.s.

Define
\begin{equation*}
F\colon Y\to K,
\qquad
F\left(x,(k_g)_{g\in G}\right)=k_e.
\end{equation*}
It remains to check that $F(S_hy)F(y)^{-1}$ is Haar distributed in $K$ for
every $h\in H$, but $F(y)$ is not Haar distributed in $K$.

For bounded Borel measurable functions $a,b\colon K\to\mathbb C$, write
\begin{equation}\label{eq:convolution}
(a*b)(r)=\int_Ka(rv^{-1})b(v)\dd m_K(v),
\qquad
\check{a}(u)=a(u^{-1}).
\end{equation}
If independent $K$-valued random variables $A$ and $B$ have densities $a$
and $b$, respectively, then $AB$ has density $a*b$, while $A^{-1}$ has
density $\check{a}$.

Put $f_\sigma=1+\sigma\psi$, for $\sigma\in\{-1,1\}$.  Since $\psi$ has Haar mean zero,
\begin{equation}\label{eq:haar-annihilates}
1*1=1,
\qquad
1*\check{\psi}=\psi*1=0.
\end{equation}
Fix $h\in H$.  Since $H$ is not vdC in $G$, we have $e\notin H$, and hence
$h\neq e$.  For fixed $x\in X$, consider the distribution of
$F(S_hy)F(y)^{-1}$ when the $K^G$-coordinate of
$y\in\{x\}\times K^G$ has distribution $\nu_x$.  The variables $k_h$ and
$k_e$ are independent, with densities $f_{s(T_hx)}$ and $f_{s(x)}$.
Therefore, the conditional density of $
F(S_hy)F(y)^{-1}=k_hk_e^{-1}$
is
\begin{align*}
f_{s(T_hx)}*\check{f}_{s(x)}
&=(1+s(T_hx)\psi)*(1+s(x)\check{\psi})\\
&=1+s(T_hx)s(x)(\psi*\check{\psi}).
\end{align*}
After integrating over $x$, \eqref{eq:sign-witness} shows that the density
of $F(S_hy)F(y)^{-1}$ is identically $1$.  Thus $F(S_hy)F(y)^{-1}$ is Haar
distributed for every $h\in H$.

On the other hand, by \eqref{eq:sign-witness}, the marginal distribution
$F_*\nu$ has density
\begin{equation*}
\int_Xf_{s(x)}\dd\mu(x)=1+m\psi\not\equiv1.
\end{equation*}
Hence $F$ is not Haar distributed, which completes the proof.
\end{proof}

%%%%%%%%%%%%%%%%%%%%%%%%%%%%%%%%%%%%%%%%%%%%%%%%%%%%%%%%%%%%%%%%%%%%%%%%%%%%%%%%%%%%%%%%%%%%%%%%%%%%%%%%%%%%%%%%%%%%%%%%%%%%%%%%%%%%%%%%%%
\subsection{\texorpdfstring{$K$}{K}-vdC sets via uniform distribution}

The original definition of $K$-vdC sets was given in terms of uniform distribution of sequences in $K$.
In this subsection we extend this definition to the setting of countably infinite amenable groups, and then show that it is the same as the measure-theoretic definition of the previous subsection.

\begin{definition}\label{def:Sequence(KG)-vdC}
    Let $G$ be a countably infinite amenable group and $\mathcal{F} = (F_n)_{n = 1}^\infty$ a left-F\o lner sequence in $G$.
    A sequence $(x_g)_{g \in G}$ taking values in $K$ is \textbf{$\mathcal{F}$-uniformly distributed (u.d.)} if for every $f \in C(K)$ we have

    \begin{equation*}
        \lim_{n\rightarrow\infty}\frac{1}{|F_n|}\sum_{g \in F_n}f(x_g) = \int_Kf\dd m_K.
    \end{equation*}
    We say that $H$ is \textbf{$(K,G,\mathcal{F})$-vdC} if for every sequence $(x_g)_{g \in G}$ taking values in $K$ for which $(x_{hg}x_g^{-1})_{g \in G}$ is $\mathcal{F}$-u.d. for every $h \in H$, we have that $(x_g)_{g \in G}$ is itself $\mathcal{F}$-u.d.
\end{definition}

\begin{theorem}\label{thm:Sequence(KG)vdCIsMeasureTheoretic(KG)vdC}
    Let $G$ be a countably infinite amenable group and $\mathcal{F} = (F_n)_{n = 1}^\infty$ a left-F\o lner sequence in $G$.
    A set $H \subseteq G$ is $(K,G,\mathcal{F})$-vdC if and only if it is measure-theoretic $(K,G)$-vdC.
\end{theorem}

\begin{proof}
We first prove that $(K,G,\mathcal F)$-vdC implies measure-theoretic
$(K,G)$-vdC.
Suppose that $H$ is $(K,G,\mathcal F)$-vdC.
Let $(X,\mathscr{B},\mu,(T_g)_{g\in G})$ be an m.p.s., and let $F:X\to K$ be Borel measurable. 
Assume that
\begin{equation}\label{eq:measure-increments-Haar}
    x\longmapsto F(T_hx)F(x)^{-1}
\end{equation}
is Haar distributed for every $h\in H$.

Let $C\subseteq[0,1]$ be the middle-thirds Cantor set.
Since $K$ is compact and metrizable, Theorems 4.18 and 18.18 of \cite{KechrisDescriptiveSetTheory} tell us that there is a continuous surjection $q:C\rightarrow K$ and a Borel measurable section $\sigma:K\rightarrow C$, i.e., $q\circ\sigma = \text{Id}_K$.

Consider $\widetilde{F}:X\rightarrow C$ given by $\widetilde{F} := \sigma\circ F$.
By the inverse Furstenberg correspondence principle \cite[Theorem~2.7]{SaulInverseFurstenberg}, there is a sequence $(c_g)_{g\in G}$ in $C$ such that, for every $j\in\mathbb N$, $h_1,\ldots,h_j\in G$, and every continuous $p:C^j\to\mathbb C$, we have
\begin{equation}\label{eq:Cantor-inverse-correspondence}
\begin{split}
    \lim_{n\to\infty}\frac{1}{|F_n|}
    \sum_{g\in F_n}
    p(c_{h_1g},\ldots,c_{h_jg})
    &=
    \int_X
    p\bigl(
        \widetilde F(T_{h_1}x),\ldots,
        \widetilde F(T_{h_j}x)
    \bigr)\dd\mu(x).
\end{split}
\end{equation}
Set $x_g = q(c_g)$ for all $g \in G$.
Fix $h\in H$ and $\varphi\in C(K)$.
Apply \eqref{eq:Cantor-inverse-correspondence} with $j=2$, $h_1=h$, $h_2=e$, and $p(u,v) = \varphi\left(q(u)q(v)^{-1}\right)$.
This gives
\begin{align*}
    \lim_{n\to\infty}\frac{1}{|F_n|}
    \sum_{g\in F_n}
    \varphi(x_{hg}x_g^{-1})& = \int_X
    \varphi\left(
        q(\widetilde F(T_hx))
        q(\widetilde F(x))^{-1}
    \right)\dd\mu(x) =
    \int_X
    \varphi\left(
        F(T_hx)F(x)^{-1}
    \right)\dd\mu(x)\\
    &= \int_K\varphi\dd m_K,
\end{align*}
where the final equality follows from
\eqref{eq:measure-increments-Haar}.
Hence $(x_{hg}x_g^{-1})_{g\in G}$ is $\mathcal F$-u.d. in $K$ for every $h\in H$.

Since $H$ is $(K,G,\mathcal F)$-vdC, it follows that
$(x_g)_{g\in G}$ is itself $\mathcal F$-u.d.  On the other hand,
applying \eqref{eq:Cantor-inverse-correspondence} with $j=1$,
$h_1=e$, and $p(u)=\varphi(q(u))$, we obtain
\begin{equation*}
    \int_X\varphi(F(x))\dd\mu(x) = \lim_{n\to\infty}\frac{1}{|F_n|}\sum_{g\in F_n}\varphi(x_g) = \int_K\varphi \dd m_K.
\end{equation*}
Since $\varphi \in C(K)$ was arbitrary, we have $F_*\mu=m_K$, as desired.

We now prove the converse.
Suppose that $H$ is measure-theoretic $(K,G)$-vdC.
Let $(x_g)_{g\in G}$ be a sequence in $K$ such that $\left(x_{hg}x_g^{-1}\right)_{g\in G}$ is $\mathcal F$-u.d. in $K$ for every $h\in H$.

Let $Z = K^G$ with its product topology, and define the right-shift action
$(R_a)_{a\in G}$ on $Z$ by $(R_a\mathbf y)_b=y_{ba}$ for all $\mathbf y=(y_b)_{b\in G}\in Z$.
Regard $\mathbf x=(x_g)_{g\in G}$ as a point of $Z$, and define
\begin{equation*}
    \nu_n
    =
    \frac{1}{|F_n|}
    \sum_{g\in F_n}\delta_{R_g\mathbf x}.
\end{equation*}
Since $Z$ is a compact metric space, the space of Borel probability measures on $Z$ is weakly sequentially compact, so let $\nu$ be an arbitrary weak cluster point of $(\nu_n)$, say $\nu_{n_j}\rightarrow\nu$.
Since $\mathcal{F}$ is a left-F\o lner sequence, $\nu$ is a $R$-invariant probability measure.

Let $\pi_e:Z\rightarrow K$ be the identity-coordinate map, i.e., $\pi_e(\mathbf y)=y_e$.
For $h\in G$, define
\begin{equation*}
    \Delta_h(\mathbf y)
    =
    \pi_e(R_h\mathbf y)\pi_e(\mathbf y)^{-1}
    =
    y_hy_e^{-1},
\end{equation*}
and observe that both $\pi_e$ and $\Delta_h$ are continuous.
Consequently, for every $\varphi\in C(K)$ and $h\in H$, we have
\begin{align*}
    \int_Z\varphi(\Delta_h(\mathbf y))\dd\nu(\mathbf y)
    &=
    \lim_{j\to\infty}
    \int_Z\varphi(\Delta_h(\mathbf y))\dd\nu_{n_j}(\mathbf y) =
    \lim_{j\to\infty}
    \frac{1}{|F_{n_j}|}
    \sum_{g\in F_{n_j}}
    \varphi\left(
        \Delta_h(R_g\mathbf x)
    \right)\\
    &=
    \lim_{j\to\infty}
    \frac{1}{|F_{n_j}|}
    \sum_{g\in F_{n_j}}
    \varphi(x_{hg}x_g^{-1}) =
    \int_K\varphi\dd m_K.
\end{align*}
We now see that $(\Delta_h)_*\nu = m_K$ for all $h \in H$, so we may apply the measure-theoretic $(K,G)$-vdC property to the map $\pi_e:Z\to K$ to see that $(\pi_e)_*\nu = m_K$.

Finally, set
\begin{equation*}
    \lambda_n
    =
    \frac{1}{|F_n|}
    \sum_{g\in F_n}\delta_{x_g}
    =
    (\pi_e)_*\nu_n.
\end{equation*}
Since $(\pi_e)_*$ is weakly continuous, and every weak cluster point $\nu$ of $(\nu_n)_{n = 1}^\infty$ satisfies $(\pi_e)_*\nu = m_K$, we see that $\lambda_n\to m_K$ weakly.
Therefore $H$ is $(K,G,\mathcal F)$-vdC.
\end{proof}

\begin{remark}\label{rem:MetrizabilityAssumption}
    As discussed in Chapters 3 and 4 of \cite{KuipersAndNiederreiter} as well as \cite{LosertUDSequencesInCompactGroups}, the compact group $K$ admits a uniformly distributed sequence if and only if $K$ is separable.
    Consequently, if $K$ is a nonseparable compact group, then every nonempty $H \subseteq G$ is vacuously a $(K,G,\mathcal{F})$-vdC set.
Separability of a compact group does not imply metrizability: for example, both the product group $(\mathbb{Z}/2\mathbb{Z})^{\mathbb{R}}$ and the Bohr compactification $b\mathbb{Z}$ are separable and nonmetrizable. 
We believe that if one is willing to work with non-standard probability spaces $(X,\mathscr{B},\mu)$ (cf. \cite{JamneshanTaoMeasureTheory}), then all of the results of this section extend to the setting in which $K$ is a compact separable group, but we do not pursue this here.
\end{remark}
%%%%%%%%%%%%%%%%%%%%%%%%%%%%%%%%%%%%%%%%%%%%%%%%%%%%%%%%%%%%%%%%%%%%%%%%%%%%%%%%%%%%%%%%%%%%%%%%%%%%%%%%%%%%%%%%%%%%%%%%
\section{vdC sets are sets of operatorial recurrence}\label{sec:vdCIsOperatorRecurrent}

\subsection{Normal random vectors and the Gaussian measure space construction}\label{ssec:GMSC}
Suppose that $\nu$ is a probability measure on $\mathbb{R}^d$ and $\mathbf{x}$ is a random variable taking values in $\mathbb{R}^d$. We use the standard notation $\mathbf{x}\sim\nu$ to mean that the random
vector $\mathbf{x}$ has distribution (or law) $\nu$; equivalently,
$\mathbb{P}(\mathbf{x}\in A)=\nu(A)$ for every Borel set $A$. If
$\mathbf{x}$ and $\mathbf{y}$ are random vectors with values in the same
Euclidean space, the notation $\mathbf{x}\sim\mathbf{y}$ means that they have the same distribution.

We begin by recalling some basic facts about normal random vectors.
Let $\mathcal{N}(0,1)$ denote the normal distribution on $\mathbb{R}$ with mean
$0$ and variance $1$.
A random vector
\begin{equation*}
\mathbf{z}=
\begin{pmatrix}
z_1&\cdots&z_d
\end{pmatrix}^{\!\top}
\end{equation*}
in $\mathbb{R}^d$ is called a \emph{standard normal random vector} if its
coordinates are independent and each $z_i$ has distribution
$\mathcal{N}(0,1)$.

Let $A\in\mathcal{M}_d(\mathbb{R})$ and set $\Sigma=AA^{\top}$.  If
$\mathbf{z}$ is a standard normal random vector in $\mathbb{R}^d$, then
\begin{equation*}
\mathbf{x}=A\mathbf{z}
\end{equation*}
is called a \emph{centered normal random vector} with distribution
$\mathcal{N}_d(0,\Sigma)$. The notation $\mathbf{x}\sim\mathcal{N}_d(0,\Sigma)$ reflects the facts that $\mathbb{E}[\mathbf{x}]=0$ and that $\Sigma$ is the covariance matrix of $\mathbf{x}$. Indeed,
\begin{align*}
\operatorname{Cov}(\mathbf{x})
:=\bigl(\operatorname{Cov}(x_i,x_j)\bigr)_{i,j=1}^d
=\mathbb{E}\bigl[\mathbf{x}\mathbf{x}^{\top}\bigr]
=A\mathbb{E}\bigl[\mathbf{z}\mathbf{z}^{\top}\bigr]A^{\top}
=AI_dA^{\top}
=AA^{\top}
=\Sigma.
\end{align*}

The notation $\mathcal{N}_d(0,\Sigma)$ is unambiguous: the distribution of
$A\mathbf{z}$ is the same for every matrix $A\in\mathcal{M}_d(\mathbb{R})$ satisfying $AA^{\top}=\Sigma$.
Indeed, if $A,B\in\mathcal{M}_d(\mathbb{R})$ satisfy $AA^{\top}=BB^{\top}=:\Sigma$, the orthogonal equivalence theorem for Gram matrices yields an orthogonal matrix $U$ such that $A=BU$ (see for example \cite[Theorem~7.3.11]{HornJohnson}). Therefore, as the standard normal distribution is invariant under linear isometries,
\begin{equation*}
A\mathbf{z}=BU\mathbf{z}\sim B\mathbf{z}.
\end{equation*}

The next lemma gives a concrete product-space version of the Gaussian measure
space construction; see \cite[Chapter~3.11]{Glasner} and
\cite[Chapter~8.2]{CornfeldFominSinai} for general treatments.

\begin{lemma}[Gaussian realization of a positive-definite function]\label{ExplicitGNSConstruction}
Let $G$ be a countably infinite group with identity element $e$, and let
$\phi\colon G\to\mathbb{R}$ be
positive definite.
There exists a unique Borel probability measure $\mu_\phi$ on $\mathbb{R}^G$ such that, if $x_g\colon\mathbb{R}^G\to\mathbb{R}$ denotes the $g$th coordinate map,
then, for every $n \in \mathbb{N}$ and every $g_1,\ldots,g_n\in G$,
\begin{equation*}
\begin{pmatrix}
x_{g_1}&\cdots&x_{g_n}
\end{pmatrix}^{\!\top}
\sim
\mathcal{N}_n\left(0,
\bigl(\phi(g_i^{-1}g_j)\bigr)_{i,j=1}^n\right).
\end{equation*}
In particular,
\begin{equation*}
\int_{\mathbb{R}^G}x_gx_k\dd\mu_\phi
=\phi(g^{-1}k)
\qquad(g,k\in G).
\end{equation*}
Moreover, $\mu_\phi$ is invariant under the shift action
$(T_a)_{a\in G}$ defined by $(T_a\mathbf{x})_g=x_{a^{-1}g}$, for all $a,g\in G,\ \mathbf{x}\in\mathbb{R}^G$.
\end{lemma}

In what follows, we will need to compute $\mathbb{P}(x_1x_2>0)$ for a
centered normal random vector $\mathbf{x}=(x_1,x_2)^\top$ in $\mathbb{R}^2$.

\begin{lemma}
\label{IntegralOfSignForGaussian}
Let $\mathbf{x}=(x,y)^\top$ be a centered normal random vector in
$\mathbb{R}^2$ with covariance matrix
$\left(\begin{smallmatrix}a&b\\ b&d\end{smallmatrix}\right)$, where $a,d>0$.
Then,
\begin{equation*}
\mathbb{P}(xy>0)
=\frac12+\frac{1}{\pi}\arcsin\left(\frac{b}{\sqrt{ad}}\right).
\end{equation*}
\end{lemma}

\begin{proof}
By symmetry,
$\mathbb{P}(xy>0)=2\mathbb{P}(x>0,y>0)$, and Sheppard's formula (see \cite[Section~2.2.1]{GenzBretz}) gives
\begin{equation*}
\mathbb{P}(x>0,y>0)
=\frac14+\frac{1}{2\pi}\arcsin\left(\frac{b}{\sqrt{ad}}\right).\qedhere
\end{equation*}
\end{proof}

\subsection{VdC implies operatorial recurrence}
Let $G$ be a countably infinite group.

\begin{theorem}\label{Thm:vdCiffOpRec}
For every nonempty subset $H$ of $G$, the following are equivalent:
\begin{enumerate}[label=\alph*)]
 \item\label{Thm:vdCiffOpRec1}
    ($H$ is vdC): for every m.p.s. $(X,\mathscr{B},\mu,(T_g)_{g\in G})$ and every
    $f\in L^\infty(X,\mu)$,
    \begin{equation*}
        \int_X f(T_hx)\overline{f(x)}\dd\mu(x)=0
        \quad\text{for every }h\in H
        \quad\Longrightarrow\quad
        \int_X f\dd\mu=0.
    \end{equation*}
    \item\label{Thm:vdCiffOpRec2} ($H$ is a set of operatorial recurrence): for every m.p.s. $(X,\mathscr{B},\mu,(T_g)_{g\in G})$ and every
    $f\in L^2(X,\mu)$,
    \begin{equation*}
        \int_X f(T_hx)\overline{f(x)}\dd\mu(x)=0
        \quad\text{for every }h\in H
        \quad\Longrightarrow\quad
        \int_X f\dd\mu=0.
    \end{equation*}
    \item\label{Thm:vdCiffOpRec3} There do not exist a positive-definite function
    $\varphi\colon G\to\mathbb{C}$ and a constant $r<0$ such that \begin{equation*}
        \varphi(h)=r
        \qquad\text{for every }h\in H.
    \end{equation*}
\end{enumerate}
\end{theorem}

\begin{remark}\label{rem:NontrivialityOfLInfinityVersusL2}
    Observe that the coordinate projection maps $x_g$ in \Cref{ExplicitGNSConstruction} are unbounded functions.
    Farhangi and Tucker-Drob asked in \cite[Question 3.7]{FTD} if a version of \Cref{ExplicitGNSConstruction} holds using only bounded functions, as that would show that every vdC set is a set of operatorial recurrence.
    A detailed discussion about the negative answer to \cite[Question 3.7]{FTD} will be the subject of another paper.
    For now, we only mention that there exists a positive definite sequence $\phi$ on the free group $F_2$ for which $\phi(e) = 1$, but for every measure-preserving system $(X,\mathscr{B},\mu,(T_g)_{g \in F_2})$ and every $f \in L^2(X,\mu)$ satisfying $\int_XT_gf\overline{f}\dd\mu = \phi(g)$ for all $g \in F_2$, we have $f \notin L^\infty(X,\mu)$.
\end{remark}

\begin{proof}
The implication
$\ref{Thm:vdCiffOpRec2}\Rightarrow\ref{Thm:vdCiffOpRec1}$ is immediate
from the inclusion $L^\infty(X,\mu)\subseteq L^2(X,\mu)$. 

We now prove
$\neg\ref{Thm:vdCiffOpRec2}\Rightarrow
\neg\ref{Thm:vdCiffOpRec3}$.
Suppose that there exist an m.p.s.
$(X,\mathscr{B},\mu,(T_g)_{g\in G})$ and a function $f\in L^2(X,\mu)$
such that
\begin{equation*}
\int_X f(T_hx)\overline{f(x)}\dd\mu(x)=0
\quad\text{for every }h\in H,
\qquad
a:=\int_X f\dd\mu\neq0.
\end{equation*}
Write $f=a\,\mathbbm{1}_X+f_0$, where $\int_X f_0\dd\mu=0$.
The correlation function $\varphi\colon G\to\mathbb{C}$ defined by
\begin{equation*}
\varphi(g):=\int_X f_0(T_gx)\overline{f_0(x)}\dd\mu(x)
\end{equation*}
is positive definite.
Moreover, for every $h\in H$, using that $\int_Xf_0\dd\mu=0$, we have
\begin{align*}
0=\int_X f(T_hx)\overline{f(x)}\dd\mu(x)
=\int_X (f_0(T_hx)+a)\overline{(f_0(x)+a)}\dd\mu(x)
=\int_X f_0(T_hx)\overline{f_0(x)}\dd\mu(x)+|a|^2,
\end{align*}
that is, $\varphi(h)=-|a|^2$, so condition \ref{Thm:vdCiffOpRec3} fails, as desired.

Finally, we prove $\neg$\ref{Thm:vdCiffOpRec3}$\implies\neg$\ref{Thm:vdCiffOpRec1}.
Let $\phi:G\to\mathbb{C}$ be positive definite and suppose that for some $r<0$ we have $\phi(h)=r$ for all $h\in H$.
We can suppose $\phi$ is real-valued; otherwise take its real part, which is still positive definite.
Moreover, $\phi(e)>0$, since otherwise positive definiteness would force $\phi$ to vanish identically.

Now consider $\mathbb{R}^G$, whose points we denote by
$\mathbf{x}=(x_g)_{g\in G}$, with its Borel $\sigma$-algebra and the shift
action $(T_a)_{a\in G}$ defined by $(T_a\mathbf{x})_g=x_{a^{-1}g}$, for all
$a,g\in G$ and $\mathbf{x}\in\mathbb{R}^G$.

\Cref{ExplicitGNSConstruction} gives us a Borel probability measure
$\mu_\phi$ on $\mathbb{R}^G$ such that, for all $g\in G$, the joint
distribution of the pair $(x_e,x_g)\in\mathbb{R}^2$ is a centered normal
random vector with covariance matrix
\begin{equation*}
\begin{pmatrix}
\phi(e)&\phi(g)\\
\phi(g)&\phi(e)
\end{pmatrix}.
\end{equation*}
In particular, for every $h\in H$, the pair $(x_{h^{-1}},x_e)$ has covariance matrix $\left(\begin{smallmatrix}a&r\\r&a\end{smallmatrix}\right)$, where $a:=\phi(e)>0$.

Now consider the sign function $s:\mathbb{R}\to\{-1,1\}$ given by $s(x)=1$ if $x\geq0$ and $s(x)=-1$ if $x<0$, and let $f_0:\mathbb{R}^G\to\{-1,1\}$ be given by $f_0((x_g)_{g\in G})=s(x_e)$.
We then have
\begin{equation*}
\int_{\mathbb{R}^G}f_0\dd\mu_\phi
=
\int_{\mathbb{R}^G}s(x_e)\dd\mu_\phi(\mathbf{x})=0,
\end{equation*}
since $x_e$ has a centered, nondegenerate normal distribution and is therefore
symmetric about $0$.
On the other hand, by \Cref{IntegralOfSignForGaussian}, for all $h\in H$ we have
\begin{align*}
\int_{\mathbb{R}^G} f_0(T_h\mathbf{x})f_0(\mathbf{x})\dd\mu_\phi(\mathbf{x})
&=\int_{\mathbb{R}^G}s(x_{h^{-1}})s(x_e)\dd\mu_\phi(\mathbf{x}) = 2\mathbb{P}(x_{h^{-1}}x_e>0)-1\\
&= \frac{2}{\pi}\arcsin\left(\frac{r}{a}\right) =:\lambda<0.
\end{align*}

To turn these negative correlations into zero correlations, we adjoin a fixed
point $*$ to $\mathbb{R}^G$. More precisely, let $
Y:=\mathbb{R}^G\sqcup\{*\}$,
extend the action to $Y$ by setting $S_g|_{\mathbb{R}^G}=T_g$ and
$S_g(*)=*$, and equip $Y$ with the probability measure
\begin{equation*}
\nu:=\frac{1}{1-\lambda}\mu_\phi
+\frac{-\lambda}{1-\lambda}\delta_*.
\end{equation*}
Finally, define $f\colon Y\to\{-1,1\}$ by
$f|_{\mathbb{R}^G}=f_0$ and $f(*)=1$. Then we are done, as for every $h\in H$,
\begin{align*}
\int_Y f(S_hy)\overline{f(y)}\,d\nu(y)
&=\frac{\lambda}{1-\lambda}
  +\frac{-\lambda}{1-\lambda}=0,
\end{align*}
whereas
\begin{equation*}
\int_Y f\,d\nu
=\frac{-\lambda}{1-\lambda}>0.\qedhere
\end{equation*}
\end{proof}

\begin{remark}
In the proof of
$\neg\ref{Thm:vdCiffOpRec3}\Rightarrow\neg\ref{Thm:vdCiffOpRec1}$, one
can obtain a counterexample on the original Gaussian m.p.s.
$(\mathbb{R}^G,\mathscr{B}(\mathbb{R}^G),\mu_\phi,(T_g)_{g\in G})$, without
adjoining a fixed point, by defining
$f_0(\mathbf{x})=s(x_e+\kappa)$ for a suitable $\kappa>0$. Indeed, the
correlations
\begin{equation*}
\int_{\mathbb{R}^G}f_0(T_h\mathbf{x})f_0(\mathbf{x})\dd\mu_\phi(\mathbf{x})
\qquad(h\in H)
\end{equation*}
have a common value that depends continuously on $\kappa$, is equal to
$\lambda<0$ when $\kappa=0$, and tends to $1$ as $\kappa\to\infty$.
Thus this value is zero for some $\kappa>0$, while
$\int_{\mathbb{R}^G}f_0\dd\mu_\phi>0$.
\end{remark}

%%%%%%%%%%%%%%%%%%%%%%%%%%%%%%%%%%%%%%%%%%%%%%%%%%%%%%%%%%%%%%%%%%%%%%%%%%%%%%%%%%%%%%%%%%%%%%%%%%%%%%%%%%%%%%%%%%%%%%%%
\bibliographystyle{alpha}
\begin{center}
	\bibliography{references}
\end{center}
\end{document}